\documentclass[12pt]{article} 

\usepackage{amsmath}
\usepackage{amsfonts}\usepackage{amsthm}
\usepackage{amssymb}   
\usepackage{latexsym}
\usepackage{euscript}    

\newcommand{\msf}[1]{\mathsf {#1}}

\newcommand{\mcal}[1]{{\mathcal {#1}}} 

\newtheorem{theorem}{Theorem}  
\newtheorem{lemma}[theorem]{Lemma}

\newtheorem{proposition}[theorem]{Proposition}

\theoremstyle{definition}
\theoremstyle{definition}
\theoremstyle{definition}
\theoremstyle{definition}
\theoremstyle{definition}\newtheorem{example}[theorem]{Example}
\theoremstyle{definition}

\theoremstyle{definition}
\theoremstyle{definition}

\theoremstyle{plain}\newtheorem*{theorem*}{Theorem}
\theoremstyle{plain}\newtheorem*{corollary*}{Corollary}
\theoremstyle{remark}\newtheorem*{remark*}{Remark}
\theoremstyle{remark}\newtheorem*{remarks*}{Remarks}
\theoremstyle{definition}\newtheorem*{conjecture*}{Conjecture}
\theoremstyle{definition}\newtheorem*{definition*}{Definition}
\theoremstyle{definition}\newtheorem*{example*}{Example}

\theoremstyle{definition}\newtheorem*{question*}{Question}
\theoremstyle{definition}\newtheorem*{questions*}{Questions}

\theoremstyle{definition}\newtheorem*{hypothesis*}{Hypothesis}

\def\qed{\ifhmode\unskip\nobreak\fi\ifmmode\ifinner\else\hskip5pt\fi\fi
 \hfill\hbox{\hskip5pt\vrule width4pt height6pt depth1.5pt\hskip1pt}}
 
\newcommand{\ov}[1]{\ensuremath{\overline{#1}}}

\newcommand{\Int}{\ensuremath{\operatorname{\mathsf{Int}}}}

\newcommand{\dist}{\ensuremath{\operatorname{\mathsf{dist}}}}

\newcommand{\RR}{\ensuremath{{\mathbb R}}}     
\newcommand{\R}[1]{\ensuremath{{\mathbb R}^{#1}}} 

\newcommand{\om}[1]{\ensuremath{\omega(#1)}}
\newcommand{\emp}{\ensuremath{\emptyset}}

\def\d#1dt{\frac{d#1}{dt}}    

\newcommand{\Gam}{\ensuremath{\Gamma}}

\newcommand{\Fix}[1]{\ensuremath{\operatorname{\mathsf {Fix}}(#1)}}
\newcommand{\co}{\colon\thinspace} 

\newcommand{\p}{\ensuremath{\partial}}

   \usepackage{txfonts}

\def\emp{\varnothing}
\reversemarginpar 	

	\def\mylabel#1{\label{#1}} 

\newcommand{\Z}[1]{\ensuremath{\operatorname{{\msf  Z}} (#1)}}
 
\newcommand{\V}{\mcal V}
\newcommand{\A}{\mcal A}
\newcommand{\B}{\mcal B}

\newcommand{\Vom}[1]{\ensuremath{\operatorname{{\mcal V}^\omega} (#1)}}

\renewcommand{\dim} {\ensuremath{{\mathsf {dim}\,}}}

\renewcommand{\om}{\omega}

\def\emp{\varnothing}

\begin{document}

\title{\bf Zero sets of Abelian Lie algebras of vector fields} 
\author{{\bf Morris W. Hirsch}
  \\ Mathematics Department\\ University of Wisconsin at
  Madison\\ University of California at Berkeley}
\maketitle

\begin{abstract}

\noindent
Assume $M$ is a $3$-dimensional real manifold without
boundary, $\A$ is an abelian Lie algebra of analytic vector fields on
$M$, and $X\in\A$.  The following result is proved:

If $K$ is a locally maximal compact set of zeroes of $X$ and
 the Poincar\'e-Hopf index of $X$ at $K$ is nonzero, there is a
point in $K$ at which all the elements of $\A$ vanish.
\end{abstract}

\tableofcontents

\section{Introduction}   \mylabel{sec:intro}
Throughout this paper $M$ denotes a real analytic, metrizable manifold
that is connected and has finite dimension $n$, fixed at $n=3$ in the main results.  

The space of (continuous) vector fields on $M$ endowed with the
compact open topology is $\V (M)$, and $\V^r {M}$ is the subspace of
$C^r$ vector fields.  Here $r$ denotes a positive integer, $\infty$,
or $\om$ (meaning analytic); this convention is abbreviated by $1\le
r\le \om$.

The {\em zero set} of $X\in \V (M)$ is $\Z X:=\{p\in M\co X_p=0\}$.
If $\Z X =\emp$ (the empty set), $X$ is {\em nonsingular}. 
The zero set of a subset $\mcal S\subset \V (M)$ is $\Z {\mcal S}:=\bigcap_{X\in S}\Z
S$. 

A compact set $K\subset \Z X$ is a {\em block of zeros} for $X$--- called
an
{\em $X$-block} for short--- if it lies in a precompact open set
$U\subset M$ whose closure $\ov U$ contains no other zeros of $X$;
such an open set is {\em isolating} for $X$, and for
$(X, K)$.

When $U$ is isolating for $X$ there is a unique maximal open
neighborhood $\mcal N_U \subset \V (M)$ of $X$ with the following
property ({\sc Hirsch} \cite{Hirsch2015a}): 
\begin{quote} \em
If $Y\in \mcal N_U$ has only finitely many zeros in $U$,
 the Poincar\'e-Hopf index of $Y|U$ depends only on $X$ and $K$. 
\end{quote} 
This index is an integer denoted by $\msf i_K (X)$, and also by $\msf
i (X, U)$, with the latter notation implying that $U$ is isolating for
$X$.\footnote{
The index can be equivalently defined as the intersection number of $X
(U)$ with the zero section of the tangent bundle of $U$ (\cite
{Bonatti92}); and as the the fixed-point index of the time-$t$ map of
the local flow of $X|U$ for sufficiently small $t> 0$.
(\cite{Dold72, Hirsch2014, Hirsch2015c}.)}

The celebrated {\sc Poincar\'{e}-Hopf} Theorem \cite{Hopf25,
  Poincare85} connects the index to the Euler characteristic
$\chi(M)$.  A modern formulation (see {\sc Milnor} \cite
{Milnor65})  runs as follows:
\begin{theorem*}[{\sc Poincar\'{e}-Hopf}]          \mylabel{th:PH}
Assume $M$ is a compact $n$-manifold,$ X\in \V (M)$, and $\Z X\cap \p
M=\emp$.  If $X$ is tangent to $\p M$ at all boundary points, or
points outward at all boundary points then $i (X, M)=\chi (M)$.  If
$X$ points inward at all boundary points, $i (X, M)= (-1)^{n-1}\chi
(M)$.
\end{theorem*}
\noindent 
For  calculations of the index in  more general settings see  
{\sc Gottlieb} \cite{Gottlieb86}, 
 {\sc Jubin}  \cite{Jubin09},
{\sc  Morse}  \cite{Morse29}, 
{\sc Pugh} \cite{Pugh68}. 

\medskip

 The
$X$-block $K$ is {\em essential} if $\msf i_K (X)\ne 0$.   When this
 holds  every $Y\in \msf N_U (X)$ 
 has an essential block of  zeros in $U$ (Theorem \ref{th:stability}).
 If $M$ is a
 closed manifold (compact, no boundary) and $\chi (M)\ne 0$, the
 Poincar\'e-Hopf Theorem implies  $\Z X$ is an essential
 $X$-block. 

C. Bonatti's proved a remarkable extension of the
Poincar\'{e}-Hopf Theorem to certain pairs of commuting analytic
vector fields on manifolds that need not be compact:
\begin{theorem*} [{\sc  Bonatti} \cite{Bonatti92}] \mylabel{th:bonatti}
Assume $\dim M\le 4$ and $\p M=\emp$.  If $X, Y\in \Vom M$ and $[X, Y]= 0$,
then $\Z Y$ meets every essential $X$-block.\footnote
{``The demonstration of this result involves a beautiful and quite
  difficult local study of the set of zeros of $X$, as an analytic
  $Y$-invariant set.'' \quad---{\sc P.\ Molino} \cite{Molino93}}
\end{theorem*}
\noindent Related results  are in
the articles \cite{Bon-Sant15,  Hirsch2015a, Hirsch2015b, HW2000,
  Lima64, Lima65,  Plante91, Turiel03}. 

Our main result is an extension of Bonatti's Theorem:
\begin{theorem}      \mylabel{th:MAINnew}
Let $M$ be a connected $3$-manifold and $\A\subset \V^\om (M)$  an abelian
Lie algebra of analytic vector fields on $M$.  Assume  $X\in\A$ is
nontrivial and
$\Z X\cap \p M=\emp$.  If $K$ is
an essential $X$-block, then $\Z\A\cap K\ne\varnothing$.
\end{theorem}
\noindent 
The proof, in  Section \ref{sec:proof}, relies heavily on Bonatti's Theorem. 
An analog for surfaces is proved in  
{\sc Hirsch} \cite[Thm.\ 1.3] {Hirsch2015a}.

\subsubsection{Application to attractors}   \mylabel{sec:attrac}
The {\em interior} $\Int (L)$ of a subset $L\subset M$ is the union of
all open subsets of $M$ contained in $L$.
 
Fix a metric on $M$.  If  $Q\subset M$ is closed, the minimum distance
from  $z\in M$ to  points of $Q$ is denoted by
$\msf {dist}(z,Q)$. 

Let $X\in \V^1 (M)$ have local flow $\Phi$.  An {\em attractor} for
$X$ (see \cite{Akin93, Conley78, Hale88, Smale67}) is a nonempty
compact set $P\subset M$ that is invariant under $\Phi$ and has a
compact neighborhood $N\subset M$ such that 
\[ \Phi_t (N)\subset  (N)\]
and 

\begin{equation}                \label{eq:limt}
 \lim_{t\to\infty} \dist (\Phi_t (x, P),P) = 0 \ \text {uniformly
   in} \, x\in N.  
\end{equation}
Such an  $N$ can be chosen so that 

\begin{equation}                \label{eq:phitn}
t >s \ge 0 \implies \Phi_t (N)\subset \Int (\Phi_s
(N)),
\end{equation}
which is assumed from now on. 
 In addition, using a result of {\sc F.\ W.\ Wilson}
\cite[Thm{.\ }2.2]{Wilson69} we choose $N$ so that:
\begin{equation}                \label{eq:N}
\mbox{\em $N$ is a compact $C^1$ submanifold and $X$ is inwardly
  transverse to $\p N$.}\footnote{This means $X_p$ is not tangent to
  $\p N$ if $p\in \p N$.}
\end{equation}

\begin{theorem}         \mylabel{th:attrac}
 Let $M$,  $\A$ and $X$ be as in {\em Theorem \ref{th:MAINnew}}.  If 
 $P\subset M$ is  a  compact attractor for  $X$ and   $\chi (P)\ne
 0$, then $\Z \A \cap P\ne\varnothing$.
\end{theorem}
\noindent
{\em Proof.}  $P$ is a proper subset of $M$--- otherwise $M$ is a
closed 3-manifold having nonzero Euler characteristic, an
impossibility (e.g.,\ {\sc Hirsch} \cite[Thm. 5.2.5]{Hirsch76}).  Fix
$N$ as above and note that $\chi (N)\ne 0$. 

 By (\ref{eq:N}) and Poincar\'{e}-Hopf Theorem there is an essential
 $X$-block $K\subset N\,\verb=\=\,\p N$, and $K\subset P$ by
 (\ref{eq:limt}).  Standard homology theory and (\ref{eq:phitn}) imply
 that the inclusion map $P\hookrightarrow N$ induces an isomorphisms on
 singular homology, hence $\chi (N)=\chi (P) \ne 0$.

The conclusion follows from Theorem \ref{th:MAINnew} applied to the
data $M', \A', X'$:
\[
M':=N, \quad
 \A':= \big\{Y|N\co Y\in \A\big\}, \quad X':=X|N. \qquad \qed
\]

\begin{example}         \mylabel{th:example}
Denote the  inner product of $x, y\in \R 3$ by $\langle x, y\rangle$
and the norm of $x$ by $\|x\|$.  Let $B_r\subset \R 3$ denote the open
ball about the origin of radius $r >0$. 

\begin{itemize}
\item {\em Assume    $\A$ is  an abelian Lie algebra of analytic vector fields on
an open set $M\subset \R 3$ that contains  $\ov B_r$. 
Let $X\in\A$ and  $r >0$ be such 
that 
\[
 \|x\|=r \implies \langle  X_p, p \rangle < 0.
\]
Then $\Z \A \cap B_r\ne\varnothing$.}
\end{itemize}
{\em Proof.} This is a consequence of Theorem \ref{th:attrac}:
 $\ov
B_r$ contains an attractor for $X$ because $X$ inwardly transverse to
$\p \ov B_r$ and $\chi (\ov B_r)=1$.
\end{example}

\section{Background material}   \mylabel{sec:back}
\begin{lemma}[\sc {Invariance}]           \mylabel{th:invariance}
If $T, S\in \A$ then $\Z S$ is invariant under $T$.
\end{lemma}
\begin{proof} Let  $\Phi:=\{\Phi_t\}_{t\in\RR}$ and  $\Psi:=\{\Psi_s\}$ denote
  the local flows of   $T$ and $S$, respectively.   If
  $t, s\in\RR$ are sufficiently close to $0$, because $[T, S]=0$ we have 
\[
 \Phi_t \circ \Psi_s = \Psi_s\circ  \Phi_t,
\]
and
\[
\Z S=\Fix \Psi:=\bigcap_{s}\Fix{\Psi_s},
\]
where $\msf{Fix}$ denotes the fixed point set.
Suppose  $p\in \Z S$.  Then $ p\in \Fix \Psi$, and
\[
\Psi_s\circ \Phi_t (p)= \Phi_t\circ\Psi_s (p)=\Phi_t (p).
\]
Consequently $\Phi_t (p) \in \Fix {\Psi_s}$ for sufficiently small $|t|, |s|$,
implying the conclusion.
\end{proof}

A closed set $Q\subset M$ is  an {\em analytic subspace of $M$},
or 
{\em analytic in $M$,} 
provided $Q$  has a locally finite covering by
zero sets of analytic maps defined on open subsets of $M$.
This is abbreviated to  {\em analytic space} when the ambient manifold $M$ is clear
from the context.  The connected components of analytic spaces are
also analytic spaces. 

Analytic spaces have very simple local topology,   owing to the
 theorem of {\sc {\L}ojasiewicz} \cite {Lo64}:
\begin{theorem} [{\sc Triangulation}]        \mylabel{th:triang}
If  $T$ is a locally finite collection of analytic spaces in $M$,
there is a triangulation of $M$ such that each  element of $T$
is covered by a  subcomplex. 
\end{theorem}

The proof of Theorem \ref{th:MAINnew} uses 
the following folk theorem:
\begin{theorem}[\sc Stability]\mylabel{th:stability}
Assume $X\in \V (M)$ 
and $U\subset M$ is 
isolating for $X$.  
\begin{description}

\item[(a)] 
If $\msf i (X, U)\ne 0$ then $\Z X\cap U\ne\varnothing$. 

\item[(b)] If   $Y\in \V (M)$ is sufficiently close to $X$,
then  $U$ is isolating for $Y$ and   $\msf i (Y,  U)=\msf i (X, U)$.

\end{description}
\end{theorem}
 \begin{proof} See
 {\sc Hirsch} \cite[Thm.\ 3.9] {Hirsch2015a}.
 \end{proof}

Let $\Z {\mcal S}$ denote the set
of  common zeros of a subset $\mcal
  S\subset\V^\om (M)$.  
\begin{proposition}             \mylabel{th:zs}
The following conditions hold for every $\mcal S \subset\A$:
\begin{description}

\item[(a)] $\Z {\mcal S}$  is analytic in $M$.

\item[(b)] Every zero dimensional $\A$-invariant set lies in $\Z \A$.

\end{description}
\end{proposition}
\begin{proof} 
  Left to the reader.
\end{proof}

\section{Proof of Theorem \ref{th:MAINnew} }   \mylabel{sec:proof}
Recall the hypotheses of the Main Theorem:
\begin{itemize}

\item {\em $M$ is a 3-dimensional manifold, }

\item {\em $\A\subset \V^\om (M)$ is an abelian Lie algebra,} 

\item {\em $X\in\A$ is nontrivial, \ $\Z X\cap \p M=\emp$, \ and
  $K$ is an essential block of zeroes for $X$.}

\end{itemize}
The conclusion to be proved is:  $\Z A \cap K\ne\varnothing$. 
It suffices to show that $\Z A$ meets every neighborhood of $K$,
because $\Z \A$ is closed and $K$ is compact.


\paragraph{Case I:}  $\dim \A =d < \infty$.  
The special case $d \le 2$ is covered by Bonatti's Theorem.  
We proceed by induction on $d$:

\medskip
\noindent 
{\bf Induction Hypothesis}          
\begin{itemize}
\item  $\dim \A =d+1,\, d \ge 2$.

\item {\em The zero set of every  $d$-dimensional
  subalgebra of $\A$   meets $K$.}
\end{itemize}
Arguing  by contradiction, we assume {\em per contra}: 
\begin{description}

\item[(PC)] $\Z \A\cap K=\emp$. 

\end{description}
An important consequence is: 
\begin{description}

\item[(A)] $\dim K \le 2$.
\end{description}
For otherwise $\dim K =3$, which entails the contradiction that   
$X$  is trivial:  $X$ is analytic and vanishes on
a 3-simplex in the connected 3-manifold $M$.

The Stability Theorem (\ref{th:stability}) implies $X$ has a
neighborood $\mcal N_U\subset \V^\om (X)$ with the following property:

\begin{description}

\item[(B)]  $ Y\in \mcal N_U \implies U$ {\em is isolating for $Y$
  and } $\msf i (Y, U) = \msf i (X, U)\ne 0$.
\end{description}

Let $G_d (\A)$ denote $d$-dimensional Grassmann manifold of
$d$-dimensional linear 
subspaces $\B$ of $\A$; these are  abelian subalgebras.

The nonempty set 
\[
G_d (\msf N_U) :=\{\B\in G_d (\A) \co \B \cap \msf N_U\ne\varnothing\}
\]
is   open in  $G_d (\A)$, hence it is a $d$-dimensional analytic manifold. 

Bonatti's Theorem and (B) imply:
\begin{description}

\item[(C)]  $\Z \B \cap K\ne\varnothing$ for all $\B \in G_d (\msf
  N_U)$.  

\end{description}
A key topological consequence of (C) is:
\begin{description}

\item[(D)] {\em If $\B$ and $\B'$ are distinct elements of $G_d (\msf
  N_U)$, then $\Z \B \cap K$ and $\Z {\B'} \cap K$ are disjoint.}
\end{description}
This holds because $\B \cup \B'$ spans $\A$, hence (PC) implies
\[\textstyle \big(\Z \B \cap K \big ) \bigcap \big (\Z {\B'} \cap
K\big) = \Z \A \cap K=\emp.
\]  

Each set $\Z \B \cap K$ is invariant (Lemma \ref{th:invariance}) and therefore 
has positive dimension by (PC).  Moreover:

\begin{description}

\item[(E)] {\em The set \ 
$\Gam_U:=\big\{\B\in G_d (\msf N_U)\co \dim \Z \B  \cap K=2 \big\}$ \
is finite.} 
\end{description} 
For otherwise (D) implies $K$ contains an infinite sequence of
pairwise disjoint compact subsets that are 2-dimensional and analytic
in $M$.  But this is impossible by (A) and the Triangulation Theorem
\ref{th:triang}.


\medskip
 (E) shows that  $\Gam_U=\emp$ provided  $U$ is small enough. 
Therefore we can assume:

\begin{description}

\item[(F)] {\em $\dim \Z\B \cap K =1$ for all  $\B\in   G_d (\msf N_U)$.}
\end{description}

  The set
\[
 Q:=\big\{(\B, p)\in G_d (\msf N_U)\times M\co p\in \Z\B\cap K\big\}.
\] 
is analytic in $G_d (\msf N_U)\times M$ (Proposition \ref{th:zs}).  
 The natural projections 
\[ \pi_1\co Q\to G_d (\msf N_U), \quad \pi_2
 \co Q\to K
\] 
are analytic,  $\pi_1$ is surjective, 
 $\pi_2$ is injective by (D). 

The sets $\Z\B\cap K$ are therefore pairwise disjoint, and each
is a 1-dimensional analytic
subspaces of $Q$ by (F).  Therefore
\[
\dim Q  = \dim G_d (\A) + \dim (\Z \B\cap K)  \le \dim K,
\]
whence
\[
\dim Q  =  d + 1 \le 2.
\]
But this is impossible because $d \ge 2$ by the Induction
Hypothesis.  This completes the inductive proof of the Main Theorem in Case I.
\paragraph{Case II:}\  {\em $\dim \A $ \, is arbitrary. } 

Consider the  family $\mathfrak F$ of  compact subsets of $K$:
\[
\mathfrak F:= \big\{\Z {\A'}\cap K\co \A'\subset \A \ \text
{is a finite-dimensional subalgebra}\big\}.
\]
Evidently 
\[\bigcap_{S\in\mathfrak F}S\,=\, \Z \A \cap K.
\]
Case  I shows every finite subset of  $\mathfrak F$ has nonempty
intersection.  As $K$ is compact, all the elements of $\mathfrak F$
have nonmpty intersection, proving  $\Z\A \cap K\ne\varnothing$. \qed


\end{document}